\documentclass[12pt]{amsart}
\usepackage{enumerate}
\usepackage{amsthm}
\usepackage{amsmath,amssymb,latexsym}
\makeatletter
\@namedef{subjclassname@2010}{%
\textup{2010} Mathematics Subject Classification}
\makeatother
\theoremstyle{definition}
\newtheorem{theorem}{Theorem}[section]

\newtheorem{proposition}[theorem]{Proposition}
\newtheorem{corollary}[theorem]{Corollary}
\theoremstyle{definition}
\theoremstyle{definition}
\newtheorem{definition}[theorem]{Definition}

\usepackage{indentfirst}

\begin{document}
\baselineskip=17pt
\title[]{free action of finite groups on spaces of cohomology type $(0,b)$}
\author[Somorjit K Singh, Hemant Kumar Singh and Tej. B. Singh]{ Somorjit K Singh, Hemant Kumar Singh and Tej Bahadur Singh}
\address{{\bf Somorjit Konthoujam Singh},
Department of Mathematics, University of Delhi,
Delhi -- 110007, India.}
\email{Ksomorjitmaths@gmail.com}
\address{{\bf Hemant Kumar Singh}, 
Department of Mathematics, University of Delhi, Delhi- 1100o7, India.}
\email{hemantksingh@maths.du.ac.in}
\address{
{\bf Tej Bahadur Singh},
Department of Mathematics, University of Delhi,
Delhi -- 110007, India.}
\email{tbsingh@maths.du.ac.in}

\date{}

\begin{abstract} 
 Let  $G$ be a finite group acting freely on a finitistic space $X$ having cohomology type $(0,b)$ (for example, $\mathbb S^{n}\times\mathbb S^{2n}$ is a space of type $(0,1)$ and the one-point union $\mathbb S^{n}\vee \mathbb S^{2n}\vee \mathbb S^{3n}$ is a space of type $(0,0)$). It is known that a finite group $G$ which contains $\mathbb Z_{p}\oplus \mathbb Z_{p}\oplus \mathbb Z_{p}$, $p$ a prime, can not act freely on $\mathbb S^{n}\times\mathbb S^{2n}$. In this paper, we show that if a finite group $G$ acts freely on a space of type $(0,1)$, where $n$ is odd, then $G$ can not contain $\mathbb Z_{p}\oplus \mathbb Z_{p}$, $p$ an odd prime. For spaces of cohomology type $(0,0)$, we  show that every $p$-subgroup of $G$ is either cyclic or a generalized quaternion group. Moreover, for $n$ even, it is shown that $\mathbb Z_{2}$ is the only  group which can act freely on $X$.
\end{abstract}


\maketitle

\section {Introduction}
 It has been  an interesting problem in the theory  of transformation groups  to find finite groups which can  occur as the fundamental groups of the  spaces that have nice universal covering spaces, such as the $n$-sphere $\mathbb S^{n}$, $\mathbb S^{n}\times \mathbb S^{m}$, the complex projective space $\mathbb CP^{n},$ etc. This is equivalent to determining the finite groups which can act freely on these spaces and to determine the orbit spaces in those cases. The first result in this direction is due to Smith \cite{Smith}. It has been proved that every abelian subgroup of a finite group $G$ which acts freely on a sphere is cyclic . Further, it was shown by Milnor \cite{Milnor} that any element of order 2 in a finite group $G$ acting freely on a mod 2 homology $n$-sphere lies in $Z(G)$, the center of the group. It follows that every subgroup of order $p^{2}$ or $2p$, $p$ a prime, of a finite group acting freely on $\mathbb S^{n}$ is cyclic. In Madsen, Thomas and Wall \cite{Madsen}, using surgery on manifolds, it is shown that these conditions are also sufficient for the existence of a free action of $G$ on $\mathbb S^{n}.$  Thus, we have a complete solution of the problem is known in the case of $\mathbb S^{n}$.   Conner \cite{Conner} has shown that a group containing $\mathbb Z_{p}\oplus \mathbb Z_{p}\oplus \mathbb Z_{p},$ $p$ a prime, cannot act freely on $\mathbb S^{n}\times \mathbb S^{n}$. A generalization of this result for free actions of finite group on the product $\mathbb S^{n}\times \mathbb S^{m}$ was obtained by Heller \cite{Heller}. It has been shown that  a finite group $G$ which contains $\mathbb Z_{p}\oplus \mathbb Z_{p}\oplus \mathbb Z_{p},$ $p$ a prime, can not act freely on $\mathbb S^{n}\times \mathbb S^{m}$. 

\indent In this paper, we study for free actions of finite group $G$ on a space of cohomology type $\mathbb S^{n}\times \mathbb S^{2n}$ and $\mathbb S^{n}\vee\mathbb S^{2n}\vee \mathbb S^{3n}$. The cohomology structure of the fixed point set of a periodic map of odd order on the spaces of latter type was first studied by Dotzel and Singh \cite{T.B}. It has been shown that $\mathbb Z_p$ can act freely on such spaces. Further investigations of $\mathbb Z_p$, $p$ a prime, action on these spaces were done by  Dotzel and Singh \cite{Dotzel} and Pergher et al. \cite{Pergher}. Recently, using the results of  Pergher et al. \cite{Pergher}, Mattos et al. \cite{Santos} proved  Borsuk-Ulam type theorems and their parametrized versions for $\mathbb Z_{2}$-action. For free actions of finite groups on a space of cohomology type $\mathbb S^{n}\vee\mathbb S^{2n}\vee \mathbb S^{3n}$,  we show here that every $p$-subgroup of $G$ is either cyclic or a generalized quaternion group.  Moreover, for such spaces, it is proved that $\mathbb Z_{2}$ is the only  group which can act freely on $X$ when $n$ is even. Moreover, if a finite group $G$ acts freely on $\mathbb S^{n}\times \mathbb S^{2n}$, then $G$  can not contain $\mathbb Z_{p}\oplus \mathbb Z_{p}$, for all odd prime $p$. This  improve the result of Heller \cite{Heller}. All spaces considered here are assume to be finitistic: A paracompact hausorff  space is finitistic if every open covering has a finite-dimensional refinement. Moreover, we will use \v{C}ech cohomology throughout the paper.

 \section {Preliminaries}

Suppose that  a compact Lie group $G$ acts on a  space $X$. If $G \hookrightarrow E_G \rightarrow B_G$ is the universal principal $G$-bundle, then the \textit{Borel construction} on $X$ is defined as the orbit space $X_G=(X \times E_G)/G$, where $G$ acts diagonally  on the product $X\times E_G$. The projection $X \times E_G \to E_G$ gives a fibration (called the \textit{Borel fibration}) $$ X_G \longrightarrow B_G$$ with fiber $X$.
  We will exploit the Leray-Serre spectral sequence associated to the Borel fibration $X \hookrightarrow X_G \longrightarrow B_G$. The $E_2$-term  of this spectral sequence  is given by $$E_2^{k,l}=H^k(B_G;\mathcal{H}^l(X;\Lambda))$$ (where $\mathcal{H}^{l}(X;\Lambda)$ is a locally constant sheaf with stalk $H^l(X;\Lambda)$, $\Lambda$ a field)\\ and it converges to $H(X_G;\Lambda)$ as an algebra. If $\pi_1(B_G)$ acts trivially on $H^{*}(X;\Lambda)$, then the  coefficient sheaf $\mathcal{H}(X;\Lambda)$ is constant  so that $$E_2^{k,l}=H^k(B_G;\Lambda)\otimes H^l(X;\Lambda).$$
	For further details about the Leray-Serre spectral sequence, refer to  Davis and Kirk \cite{Davis} and McCleary \cite{McCleary}. For $G=\mathbb Z_p$, $p$ a prime, we take $\Lambda=\mathbb Z_p$ and write $H^{*}(X)$ to mean $H^{*}(X;\mathbb Z_p).$
	We  recall that
	\begin{eqnarray*}
\indent H^{*}(B_{G})=
     \begin{cases}
	\mathbb Z_p[t] & \text{deg}\ t=1\ \text{for}\ p=2\ \text{and}\\
	\mathbb Z_p[s,t] &\text{deg}\ s=1,\ \text{deg}\ t=2\ \text{for}\ p>2\ \text{and}\ \beta_{p}(s)=t,
	  \end{cases}
	\end{eqnarray*}
	where $\beta_p$ is the mod-$p$ Bockstein homomorphism associated with the coefficient sequence $0\rightarrow \mathbb Z_p\rightarrow\mathbb Z_{p^{2}}\rightarrow \mathbb Z_{p}\rightarrow 0$. We also recall that if $X$ is a paracompact  Hausdorff free $G$-space, $G$ a compact Lie group, then $X/G\simeq X_G.$
	Note that if  $X$ is connected  $G$-space, then $E_{2}^{*,0}=H^{*}(B_{G})$. Volovikov \cite{Volovikov} introduced the following notion of   numerical index of a $G$-space.
 \begin{definition}(\cite{Volovikov}) The index $i(X)$  is the smallest $r$ such that for some $k$, the differential $d_r: E_{r}^{k-r,r-1}\longrightarrow E_{r}^{k,0}$ in the cohomology Leray-Serre Spectral sequence of the fibration $X\stackrel{i}\hookrightarrow X_G \stackrel{\pi}{\longrightarrow} B_G$ is nontrivial.
\end{definition}
Clearly, $i(X)= r$ if $E_2^{k,0}= E_3^{k,0}=...=E_r^{k,0}$ for all $k$ and $E_r^{k,0}\not=E_{r+1}^{k,0}$ for some $k$.
If $E_2^{*,0}= E_\infty^{*,0},$ then $i(X)=\infty$.
\begin{proposition} (\cite{Volovikov}, Proposition 2.1)
If there exists an equivariant map between  $G$-spaces $X$ and $Y$, then $i(X)\leq i(Y).$
\end{proposition}
Given two integers $a$ and $b$, a  space $X$ is said to have cohomology type $(a,b)$ if $H^{i}(X,\mathbb Z)\cong \mathbb Z$ for $i=0,2n,$ and $3n$ only. Also, the generators $x\in H^{n}(X;\mathbb Z)$, $y\in H^{2n}(X;\mathbb Z)$ and $z\in H^{3n}(X;\mathbb Z)$ satisfies $x^{2}=ay$ and $xy=bz$. For example, $\mathbb S^{n}\times \mathbb S^{2n}$ has type $(0,1)$, $\mathbb CP^{3}$ and $\mathbb QP^{3}$ have type $(1,1)$,  $\mathbb CP^{2}\vee\mathbb S^{6}$ has type $(1,0)$ and $\mathbb S^{n}\vee \mathbb S^{n}\vee \mathbb S^{3n}$ has type $(0,0)$. Such spaces were first investigated by James \cite{James} and Toda \cite{Toda}. 
\begin{proposition}([16, Theorem 4.1])
 If $G=\mathbb Z_2$ acts freely on a space $X$ of cohomology type $(a,b),$ where $a$ and $b$ are even,  characterized by an  integer $n>1$, then  $$H^{*}(X/G)=\mathbb Z_2[u,w]/\langle u^{3n+1},w^2+\alpha u^{n}w+\beta u^{2n},u^{n+1}w\rangle,$$ where deg $u=1$, deg $w=n$ and $\alpha, \beta \in \mathbb Z_2.$
\end{proposition}
\begin{proposition}([19, Theorem 2])
Suppose that $G=\mathbb Z_p$, $p>2$ a prime, act freely on a space $X$ of cohomology type $(a,b),$ where  $a=0$ (mod $p)=b$. Then $$H^{*}(X/G)=\mathbb Z_p[u,v,w]/\langle u^{2},w^2,v^{\frac{n+1}{2}}w,v^{\frac{3n+1}{2}}\rangle,$$ where deg $u=1,$  deg $w=n$, $v=\beta_{p}(u)$ ($\beta_{p}$ being the mod-$p$ Bockstein) and $n$ is odd.
\end{proposition}
\section {main results}

\indent Let $X$ be a space of cohomology type $(a,b),$ characterized by an integer $n>1$, where $a=0$ (mod $p)$ and $b=0$ (mod $p)$ or $b\not=0$ (mod $p)$, $p$ a prime. We show  that  the group $G=\mathbb Z_{p}\oplus \mathbb Z_{p}$ cannot act freely on a space $X$ and, for even $n$ and $a=0$ (mod $p)=b$, we shall show that  the only finite group which acts freely on $X$ is $\mathbb Z_{2}$. We also construct a space $X$ of cohomology type $(a,b)$, where $a=0$ (mod $p)=b$, $n>1$ and an example of free involution on $X$. Recall that $G=\mathbb Z_{p}$, $p$  an odd prime, can act freely on a space $X$ of cohomology type $(0,0)$ \cite{T.B}.

\begin{theorem}
Let $X$ be a space of  cohomology type $(a,b),$ characterized by an integer $n>1.$ Then the group $G=\mathbb Z_p\oplus \mathbb Z_p$, $p>2$ a prime,  cannot act freely on $X$ if $a=0$ (mod $p).$
\end{theorem}
\begin{theorem}
Let $X$ be a space of  cohomology type $(a,b),$ characterized by an integer $n>1.$  If $a$ and $b$ are even integers, then the group $G=\mathbb Z_2\oplus \mathbb Z_2$ cannot act freely on $X$. 
\end{theorem}
We first prove the following  propositions.
\begin{proposition}
Let $X$ be a space of cohomology type $(a,b),$ characterized by an integer $n>1.$ If $G=\mathbb Z_p$, where $p>2$ a prime, acts freely on $X$ and $a=0$ (mod $p)=b$, then $n$ is odd and $i(X)=3n+1$.
\end{proposition}
\begin{proof}
To prove this proposition, we recapitulate the proof of Theorem 2 \cite{Dotzel}. Suppose $G$ acts freely on $X$. Then $n$ must be odd, otherwise $\chi(X^{G})=\chi(X)\not=0$ mod $p$ (by Floyd's Formula).  Moreover, $H^{*}(X_G)=0$ in higher degree, by [6, Theorem 1.5, p.374]. Clearly, the induced action of $G$ on $H^{*}(X)$ is trivial, so we have $E_2^{k,l}=H^{k}(B_G)\otimes H^{l}(X).$
 Let $x\in H^{n}(X), y\in H^{2n}(X)$ and $z\in H^{3n}(X)$ be the generators.  Then, $x^2=0$ and $xy=0.$  If $d_{n+1}(1\otimes x)=t^{\frac{n+1}{2}}\otimes 1$, then $d_{n+1}(1\otimes y)=0$, and we have $0=d_{n+1}((1\otimes x)(1\otimes y))=t^{\frac{n+1}{2}}\otimes y$, a contradiction. Therefore, $d_{n+1}(1\otimes x)=0$. Assume now that $d_{n+1}(1\otimes y)=0$.  And, if $d_{n+1}(1\otimes z)=0$, implies that $E_{2n+1}^{*,*}=E_{2}^{*,*}$. Further, if $d_{2n+1}(1\otimes y)=st^n\otimes 1$, then  $0=d_{n+1}((1\otimes x)(1\otimes y))=st^{n}\otimes x$, a contradiction. Therefore, $d_{2n+1}(1\otimes y)=0,$ then $E_{3n+1}^{*,*}=E_{2}^{*,*}$. In this case, at least $n$th and $2n$th lines of the spectral sequence survive to infinity, contradicting our hyphothesis. On the other hand, if $d_{n+1}(1\otimes z)=t^{\frac{n+1}{2}}\otimes y,$ then  $E_{n+2}^{k,l}=\mathbb Z_p$ for $k\geq 0$ and $l=0,n$; $E_{n+2}^{k,l}=\mathbb Z_p$  for $0\leq k\leq n$ and $l=2n$ and zero otherwise. Clearly, $E_{\infty}^{*,*}=E_{n+2}^{*,*}$, and thus the $n$th  and the bottom lines of the spectral sequence survive to infinity, a contradiction.   Therefore, we must have $d_{n+1}(1\otimes y)=t^{\frac{n+1}{2}}\otimes x.$
We have, $E_{n+2}^{k,l}=\mathbb Z_p$ for $k\geq 0$ if $l=0, 3n$, $E_{n+2}^{k,l}=\mathbb Z_p$ for $0\leq k\leq n$ if $l=n$ and zero otherwise.  So we have, $E_{3n+1}^{*,*}=E_{n+2}^{*,*}.$  Obviously, the differential $d_{3n+1}:E_{3n+1}^{0,3n}\rightarrow E_{3n+1}^{3n+1,0}$ must be nontrivial, otherwise the top and bottom lines of the spectral sequence survive to inifinity. Hence, $i(X)=3n+1.$
\end{proof}

The proof of the following proposition  is similar to the proof of previous proposition.
\begin{proposition}
Let $X$ be a space of cohomology type $(a,b),$ characterized by an integer $n>1$. If $G=\mathbb Z_2$ acts freely on $X$ and $a$ and $b$ are even integers, then  $i(X)=3n+1.$
\end{proposition}

\begin{proposition}
Suppose that $G=\mathbb Z_p,$ $p>2$ a prime, act freely on a space $X$ of cohomology tpye $(a,b)$, where $a=0$(mod $p)$ and $b\not=0$ (mod $p)$. Then $$H^{*}(X/G)=\mathbb Z_p[u,v,w]/\langle u^{2},w^2,v^{\frac{n+1}{2}}\rangle,$$ where deg $u=1,$ deg $w=2n$, $v=\beta_{p}(u)$ and $n$ is odd (Thus, $X/G\sim_{p} L_p^{n}\times \mathbb S^{2n}$). Moreover, $i(X)=n+1.$
\end{proposition}
\begin{proof} As in the Proposition 3.3, we see that $n$ is odd and $E_2^{k,l}=H^{k}(B_G)\otimes H^{l}(X).$
  Let $x\in H^{n}(X), y\in H^{2n}(X)$ and $z\in H^{3n}(X)$ be the generators.  Then, $x^2=0$ and $xy=bz,$ where $0\not=b\in \mathbb Z_p$.
Assume  that $d_{n+1}(1\otimes x)=0$.  If $d_{n+1}(1\otimes y)=t^{\frac{n+1}{2}}\otimes x$, then  $0=d_{n+1}((1\otimes y)(1\otimes y))=2(t^{\frac{n+1}{2}}\otimes xy)$, a contradiction. Therefore, $d_{n+1}(1\otimes y)=0$ and so $d_{n+1}(1\otimes z)=0$. Therefore, $E_{2n+1}^{*,*}=E_{2}^{*,*}$.  Now, if $d_{2n+1}(1\otimes y)=st^{n}\otimes 1$, then  $0=d_{2n+1}((1\otimes y)(1\otimes y))=2(st^{n}\otimes y)$, a contradiction.  On the otherhand if $d_{2n+1}(1\otimes y)=0$, then $d_{2n+1}(1\otimes z)=0.$  It is also obvious that  $d_{3n+1}(1\otimes z)=0$. Thus, in this case, spectral sequence collapses and hence $H^{*}(X_G)\not=0$ in higher degree, a contradiction.   Therefore, $d_{n+1}(1\otimes x)=t^{\frac{n+1}{2}}\otimes 1$. Then, $d_{n+1}(1\otimes y)=0$ and $d_{n+1}(1\otimes z)=t^{\frac{n+1}{2}}\otimes \frac{1}{b}y.$ We have 
\begin{eqnarray*}
\indent E^{k,l}_{\infty}=
     \begin{cases}
	\mathbb Z_p & 0\leq k\leq n\ \text{and}\  l=0,2n.\\
	0 & \mbox{otherwise.}
	  \end{cases}
	\end{eqnarray*}
	Consequently,
	\begin{eqnarray*}
\indent H^{j}(X_G)=
     \begin{cases}
	\mathbb Z_p & 0\leq j\leq n\ \text{and}\  2n\leq j\leq 3n.\\
	0 & \mbox{otherwise.}
	  \end{cases}
		\end{eqnarray*}
		Let $u =\pi^{*}(s)$ and $v =\pi^{*}(t)$ be determined by $s\otimes 1$ and $t\otimes 1$, respectively. 
Clearly, $u^2=v^{\frac{n+1}{2}}=0$. Since $1\otimes y$ is a permanent cocycle, so it determines element  $w\in H^{2n}(X_G)$ such that $i^{*}(w)=y.$ Therefore, the cohomology ring of $X_G$ is given by 
		$$\mathbb Z_p[u,v,w]/\langle u^2,v^{\frac{n+1}{2}},w^{2}\rangle,$$ where deg $u=1$, $\beta_{p}(u)=v$, deg $w=2n$
and $n$ is odd. This complets the proof.	
\end{proof}

\noindent \textbf{Proof of Theorem 3.1}: Suppose that $G=H\oplus K$, where $H= K=\mathbb Z_{p}$, acts freely on the space $X.$
Then there is a free action  of $K$ on the orbit space $Y=X/H$ via the canonical isomorphism $K\approx G/H$;  in fact, for an element $Hx=[x]$ in $Y,$ one defines $k[x]=[kx]$ for all $k\in K$. Obviously, the restriction of the action of $G$ on $X$ to $K$ is free.  With these actions of $K$ on $X$ and $Y$, the orbit map $\pi_{H}: X\rightarrow Y$ is an equivariant. So, by Propositions 2.2,  $i(X)\leq i(Y)$ and by  Propositions 3.3 and 3.5, we have $i(X)=n+1$ or $3n+1$. However, we show that  $i(Y)=2$, which contradicts the above inequality and hence the theorem. The proof of this fact is divided in two parts depending upon whether  $b=0$ (mod) $p$ or $b\not=0$ (mod) $p$.\\ 
First consider the case $b\not=0$ (mod $p)$.
By the Proposition 3.5, we have  $H^{*}(Y)=\mathbb Z_p[u,v,w]/\langle u^{2},w^2,v^{\frac{n+1}{2}}\rangle,$ where deg $u=1$,  deg $w=2n$ and $v=\beta_{p}(u)$. Thus
\begin{eqnarray*}
\indent H^j(Y)=
     \begin{cases}
	\mathbb Z_p & 0\leq j\leq n\ \text{and}\  2n\leq j\leq 3n\\
	0 & \mbox{otherwise.}
	  \end{cases}
	\end{eqnarray*}
	Clearly, induced action of $K$ on $H^{*}(Y)$ is trivial. Therefore, the $E_2$-term of the Leray-Serre spectral sequence of the fibration $Y\hookrightarrow Y_{K}\rightarrow B_K$ can be written as $E_{2}^{*,*}=H^{*}(B_K)\otimes H^{*}(Y).$  Since the action of $K$ on $Y$ is free,  $d_{r}\not=0$ for some $r\geq2$. If $d_{2}(1\otimes u)=0$, then $1\otimes u$ is a permanent cocycle. Hence, there exists a nonzero element $u^{'}\in H^1(Y_G)$ such that $i^{*}(u^{'})=u$. If $d_{2}(1\otimes v)\not=0$, then $E_{\infty}^{0,2}=E_3^{0,2}=0,$ and we see that  the homomorphism $i^{*}:H^{2}(Y_G)\rightarrow H^{2}(Y)$ is trivial. Now, by the naturality of $p$-Bockstien homomorphism, we have $v=\beta_{p}(i^{*}(u^{'}))=i^{*}(\beta_{p}(u^{'}))=0,$  a contradiction. Therefore, $d_2(1\otimes v)=0$. For the same reason, we obtained $d_{3}(1\otimes v)=0.$   Also, it is obvious that $d_{r}(1\otimes w)=0$ for $r\leq n$. Moreover, since $w^2=0,$ and deg $w$ is even, it is easily seen that $d_{r}(1\otimes w)=0$ for all $ r\geq n+1.$  Thus in this case, the spectral sequence collapses to $E_2$-term,  contrary to fact that the action of $K$ on $Y$ is free. Hence, we find that  $d_{2}(1\otimes u)\not=0$ and we have $i(Y)=2$. \\
Next, consider the case $b=0$ (mod $p)$.  By the Proposition 2.4, we have
$$H^{*}(Y)=\mathbb Z_p[u,v,w]/\langle u^{2},w^2,v^{\frac{n+1}{2}}w,v^{\frac{3n+1}{2}}\rangle,$$
where deg $u=1$,  deg $w=n$ and $v=\beta_{p}(u)$ ($\beta_{p}$ being the mod-$p$ Bockstein).
Thus
\begin{eqnarray*}
\indent H^j(Y)=
     \begin{cases}
	\mathbb Z_p & 0\leq j\leq n-1\ \text{and}\  2n+1\leq j\leq 3n\\
	\mathbb Z_p \oplus \mathbb Z_p & n\leq j\leq 2n\\
	0 & \mbox{otherwise.}
	  \end{cases}
	\end{eqnarray*}
We observe that the action of $K$ induced on $H^{*}(Y)$ is trivial.
Let $g$ be a generator of $K=\mathbb Z_p$. By naturality of cup product, we get   $g^{*}(uv^{j}w)=g^{*}(u)g^{*}(v^{j})g^{*}(w)$ and $g^{*}(v^{j})=(g^{*}(v))^{j}$. Clearly $g^{*}(u)=u$ and $g^{*}(v)=v.$ If the induced action of $K$ is nontrivial, we get $g^{*}(w)=uv^{\frac{n-1}{2}}$ or $uv^{\frac{n-1}{2}}+w.$ If $g^{*}(w)=uv^{\frac{n-1}{2}},$ then $w=g^{*p}(w)=g^{*p-1}(uv^{\frac{n-1}{2}})=uv^{\frac{n-1}{2}},$  a contradiction. If $g^{*}(w)=uv^{\frac{n-1}{2}}+w,$ then $0=g^{*}(v^{\frac{n+1}{2}}w)=v^{\frac{n+1}{2}}(uv^{\frac{n-1}{2}}+w)=uv^{n},$ which is again a contradiction. Therefore, it follows that the induced action of $K$ on $H^{*}(Y)$ is trivial.
Thus, the fibration $Y\hookrightarrow Y_K\rightarrow B_K$ has a simple  local coefficient. Thus,  $E_2^{k,l}= H^{k}(B_K)\otimes H^{l}(Y).$ As above, we see that if $d_{2}(1\otimes u)=0,$ then the spectral sequence collapses to $E_2$-term,  contrary to fact that the action of $K$ on $Y$ is free. Thus, we have $d_2(1\otimes u)\not=0$ and  $i(Y)=2.$
 \hfill{$\Box$}

\noindent \textbf{Proof of Theorem 3.2:} Suppose that $G=H\oplus K$, where $H= K=\mathbb Z_{2}$, acts freely on the space $X.$ As in case of Theorem 3.1, there is a free action of $K$ on $Y=X/H$ such that the map $\pi_{H}: X\rightarrow Y$, $x\mapsto Hx$, is $K$-equivariant  map. So $i(X)\leq i(Y).$
 By the Proposition 2.3, we have
$$H^{*}(Y)=\mathbb Z_2[u,w]/\langle u^{3n+1},w^2+\alpha u^{n}w+\beta u^{2n},u^{n+1}w\rangle,$$
where deg $u=1,$ deg $w=n$ and  $\alpha, \beta \in \mathbb Z_2$. Thus
\begin{eqnarray*}
\indent H^j(Y)=
     \begin{cases}
	\mathbb Z_2 & 0\leq j\leq n-1\ \text{and}\  2n+1\leq j\leq 3n\\
	\mathbb Z_2 \oplus \mathbb Z_2 & n\leq j\leq 2n\\
	0 & \mbox{otherwise.}
	  \end{cases}
	\end{eqnarray*}
	As in Theorem 3.1, the induced action of $K$ on $H^{*}(Y)$ is trivial. So $E_2^{k,l}= H^{k}(B_K)\otimes H^{l}(Y).$ 
Since   $K$  acts freely on $Y$, some differential $d_r: E_r^{k,l}\longrightarrow E_r^{k+r,l-r+1}$ must be nontrivial. Clearly, either $d_{2}(1\otimes u)\not=0$ or $d_2(1\otimes u)=0$ and $d_{r}(1\otimes w) \neq 0$, for some  $2\leq r\leq n+1$. In the latter case, suppose that  $d_{r}(1\otimes w)=t^{r}\otimes u^{n+1-r},$ for some $2\leq r\leq n+1.$   Then, we have $0=d_{r}((1\otimes u^{n+1})(1\otimes w))=t^{r}\otimes u^{2n+2-r},$  a contradiction. Therefore, we must have $d_{2}(1\otimes u)\not=0$ and $i(Y)=2,$ so that $i(X)\leq 2$. This contradicts the Proposition 3.4.
 \hfill{$\Box$}

\indent Now, we prove the following corollary.
\begin{corollary} Let $G$ be a finte group which acts freely on a space $X$ of cohomology type $(a,b),$ characterized by an integer $n>1.$  If $p>2$ a prime and $a=0$ (mod $p)$, then every $p$-subgroup of $G$ is cyclic.
\end{corollary}
\begin{proof}
  Let $p$ be an odd prime and  $H$  a  $p$-subgroup of a  group $G$.  Then center of $H$, $Z(H)\not=\{1\}$. Let $K\subset Z(H)$ such that $|K|=p$. If $K^{'}\subset H$ is another subgroup such that $|K^{'}|=p$, then $K\cap K^{'}=\{1\}$  so that $K\oplus K^{'}\subset H$.  By Thereom $3.1$, this is not possible, and hence there is only one subgroup of order $p$ in $H$. From (\cite{Rotman}, Theorem 5.46, p.121), it follows that every $p$-subgroup of $G$ is cyclic.
\end{proof}
Again, by the Theorem 3.2, we have the following.
\begin{corollary} Let $G$ be a  group which acts freely on a space $X$ of cohomology  type $(a,b),$ characterized by an integer $n>1,$ where $a$ and $b$ are even.
\begin{enumerate}[(I)]
\item If $G$ is finite then every $2$-subgroup of $G$ is  either cyclic or  a generalized quaternion group.
\item If $G$ is infinite then $G$ cannot contain the rotation group $SO(3)$ as a subgroup.
\end{enumerate}
\end{corollary}
Proof of Corollary 3.7 is similar to the proof of Corollary 3.6. 
\begin{theorem}
Let $X$ be a space of cohomology type $(a,b),$ characterized by an integer $n>1.$ 
If $n$ is even and $a=0$ (mod $p)=b$, $p$ a prime,  then the only finite group which  acts freely on $X$ is $\mathbb Z_2.$
\end{theorem}
\begin{proof}
 Suppose that $G$ is finite group acting freely on $X$. If $p$ is an odd prime and $p|\ |G|$ then, $\mathbb Z_{p}$ can be regarded as a subgroup of $G$, and by Flyod formula, we have $\chi(X)=\chi (X^{\mathbb Z_p})$ (mod $p)$. Since $n$ is even, we have $\chi(X)=4$ and therefore $X^{\mathbb Z_{p}}\not=\phi$, a contradiction. Therefore, $G$ contains no element of odd prime order. Hence $|G|=2^{k}$, for some interger $k\geq 1$. If $k>1,$ then either $G$ has cyclic subgroup of order $4$ or $G$ has exponent $2.$ In either case, there is a free action of $\mathbb Z_{2}$ on $X/\mathbb Z_{2}=Y.$ With the notations as in Theorem 3.2, we must have $d_{2}(1\otimes u)=t^{2}\otimes 1$. 
Since $n$ is even,  $0=d_{2}((1\otimes u)(1\otimes u^{3n}))=t^{2}\otimes u^{3n}$, a contradiction. Hence $G$ must be $\mathbb Z_{2}.$
   \end{proof}
	
	Now, we construct example of spaces of  cohomology type $(0,0)$ and show that $\mathbb Z_2$ acts freely on these spaces.

\noindent \textbf{Example}: Consider the antipodal actions of $\mathbb Z_2$ on $\mathbb S^{2n}$ and $\mathbb S^{3n}$, where $n>1$. Then $\mathbb S^{n-1}\subset \mathbb S^{2n}\cap \mathbb S^{3n}$ is invariant under this action. So, we have a free $\mathbb Z_{2}$-action on  $X=\mathbb S^{2n}\cup_{\mathbb S^{n-1}} \mathbb S^{3n},$  obtained by attaching the sphere $\mathbb S^{2n}$ and $\mathbb S^{3n}$ along $\mathbb S^{n-1}$.  Let $A=X-\{p\}$ and $B=X-\{q\}$, where $p\in \mathbb S^{2n}-\mathbb S^{n-1}$ and $q\in \mathbb S^{3n}-\mathbb S^{n-1}$. Then $A \simeq\mathbb S^{3n}$, $B \simeq\mathbb S^{2n}$ and $A\cap B\simeq \mathbb S^{n-1}.$ By Mayer-vietoris cohomology exact sequence, we have  $H^{i}(X;\mathbb Z_{p})=\mathbb Z_{p}$ for $i=0,n,2n,3n$ and trivial group otherwise. Let $x\in H^{n}(X;\mathbb Z_{p}), y\in H^{2n}(X;\mathbb Z_{p})$ and $z\in H^{3n}(X;\mathbb Z_{p})$ be generators. Obviously, $j^{*}(x)=0$ so that  $j^{*}(x^2)=0$ and $j^{*}(xy)=0$, where $j^{*}: H^{k}(X;\mathbb Z_{p})\rightarrow H^{k}(A;\mathbb Z_{p})\oplus H^{k}(B; \mathbb Z_{p})$. Since $j^{*}$ is an isomorphism for $k=2n,3n$. We have $x^2=xy=0$. Hence, $X$ is a space of  type $(a,b),$ where $a=0$ (mod $p)=b$ and $n>1.$

\bibliographystyle{amsplain}

\end{document}